\documentclass{article}

\usepackage{a4,amsmath,amsthm,amssymb,url}

\numberwithin{equation}{section}

\theoremstyle{plain}
\newtheorem{thm}{Theorem}[section]

\newtheorem{cor}[thm]{Corollary}

\newtheorem{lem}[thm]{Lemma}
\newtheorem{pro}[thm]{Problem}

\theoremstyle{definition}

\newtheorem{de}[thm]{Definition}

\usepackage{accents}

\newcommand{\algop}[2]{\langle {#1}, {#2} \rangle}

\newcommand{\setsuchthat}{\ : \ }
\newcommand{\A}{{\mathbf A}}
\renewcommand{\aa}{\bar{a}}
\newcommand{\aaa}{\bar{\bar{a}}}

\newcommand{\B}{{\mathbf B}}
\newcommand{\C}{{\mathbf C}}

\newcommand{\Clo}{\mathrm{Clo}}
\newcommand{\Con}{\mathrm{Con}}
\newcommand{\D}{{\mathbf D}}
\newcommand{\Def}{\mathrm{Def}}

\newcommand{\G}{{\mathbf G}}

\newcommand{\N}{{\mathbb{N}}}

\renewcommand{\o}{o}
\newcommand{\oo}{\bar{\o}}
\newcommand{\ooo}{\bar{\bar{\o}}}
\newcommand{\kset}{\underline{k}}

\newcommand{\n}{\underline{n}}
\newcommand{\nn}{N}
\renewcommand{\P}{\mathcal{P}}

\newcommand{\Pol}{\mathrm{Pol}}

\newcommand{\R}{{\mathcal{R}}}
\newcommand{\Ri}{{\mathbf{R}}}

\newcommand{\mtt}[4]{\left(\begin{array}{cc} {#1} & {#2} \\ {#3} & {#4} \end{array}\right)}

\newcommand{\X}{{\mathbf X}}
\newcommand{\Z}{{\mathbb{Z}}}
\newcommand{\ISP}[1]{\operatorname{\mathbb{ISP}}(#1)}
\newcommand{\ISPP}[1]{\operatorname{\mathbb{ISP^+}}(#1)}
\newcommand{\ISCP}[1]{\operatorname{\mathbb{IS}_C\mathbb{P}^+}(#1)}

\date{\today}

\begin{document}
\title{Supernilpotence prevents dualizability}

\author{Wolfram Bentz \\
\small{Centro de \'Algebra da Universidade de Lisboa} \\
\small{1649-003 Lisboa, Portugal, wfbentz@fc.ul.pt} \\ 
\and Peter Mayr \\
\small{Centro de \'Algebra da Universidade de Lisboa} \\
\small{ 1649-003 Lisboa, Portugal} \\{\footnotesize \&}\\
\small{Institute for Algebra, Johannes Kepler University Linz} \\
\small{ 4040 Linz, Austria, pxmayr@gmail.com}}

\date{}
\pagestyle{plain}

\maketitle
\begin{abstract}
We address the question of the dualizability of nilpotent Mal'cev algebras, showing that 
nilpotent
 finite Mal'cev algebras with a non-abelian supernilpotent congruence are inherently  non-dualizable.
In particular, finite
nilpotent
 non-abelian Mal'cev algebras of finite type are non-dualizable if they are direct products of algebras of prime power order.

We show that these results cannot be generalized to nilpotent algebras by giving an
example of a group expansion of infinite type that is nilpotent and non-abelian, but dualizable.
To our knowledge this is the first construction of a non-abelian nilpotent dualizable algebra.
 It has the curious property that all its non-abelian finitary reducts
with group operation are non-dualizable. We were able to prove dualizability by utilizing a new
clone theoretic approach 
developed by Davey, Pitkethly, and Willard.

Our results suggest that supernilpotence plays an important role in characterizing dualizability among Mal'cev algebras.
\vskip 2mm

\noindent\emph{$2010$ Mathematics Subject Classification\/}. 03C05, 08A05,08B05,08C15.

\vskip 2mm


\end{abstract}
\section{Introduction}
\label{sec:intro}

Natural dualities are representations of elements of an algebra as continuous structure preserving maps
(obtained in a certain natural way). For example, Stone duality represents Boolean algebras by
 Boolean spaces. A finite algebra $\A$ is dualizable if every algebra from the quasi-variety generated by $\A$ has such a representation.

Clark and Davey (\cite{CD:NDWA}, p. 291) asked the question:
``Characterize the dualizable finite algebras in a given class $\C$ of algebras", suggesting
(among other options) to let $\C$ be the class of all algebras generating a congruence permutable variety
(i.e. algebras with a Mal'cev term). Recently at the Conference on Order, Algebra, and Logics in Nashville 2007,
Ross Willard \cite{Wi:FUPC} revived this question in light of a new approach to show dualizability.

While Mal'cev algebras are in general considered to be well-understood, the question of their dualizability has so far only been addressed for various classes of classical algebras (see below for examples).
 This contrasts sharply with the situation for algebras in congruence distributive varieties, which
have been shown to be dualizable if and only if they have a near-unanimity term~\cite{DHM:NAOG}.

In this paper, we will make a start on the dualizablity problem for Mal'cev algebras by examining the role of
nilpotence. As usual, we will restrict commutator theoretic properties
to the setting of congruence-modular varieties; note that this implies that any nilpotent algebra is a Mal'cev algebra \cite[Theorem 6.2]{FM:CTFC}.

Our interest in nilpotence comes from the observation that in the previously classified classes of Mal'cev
algebras, all dualizable nilpotent algebras were in fact abelian. As we will see, this property is false
for Mal'cev algebras in general, but holds if one replaces nilpotence with the slightly stronger condition of
supernilpotence. Our main theorem can be expressed in
the usual terminology of nilpotence as follows:

\begin{thm}  \label{th:nd}
 Let $\A$ be a finite nilpotent algebra of finite type in a congruence modular variety.
 Assume that $\A$ is non-abelian and a direct product of algebras of prime power order.
 Then $\A$ is inherently non-dualizable.
\end{thm}

 This implies several known non-dualizability results, for example:
\begin{enumerate}
\item finite groups with non-abelian Sylow subgroups are not dualizable~\cite{QS:NGAN};
\item a finite ring with $1$ is not dualizable if the square of its Jacobson radical is not $0$
 (see~\cite{CI:NDQ} for commutative rings);
\item a finite ring is not dualizable if it contains a nilpotent subring that is not
 a zero-ring~\cite{Sz:FNR}.
\end{enumerate}
An application of the theorem to a class of algebras that were not previously known to be
non-dualizable are the finite non-abelian loops whose multiplication group
(the group generated by all left and right translations) is nilpotent (\cite{Ve:PGLG}, Proposition 3.2,
see also \cite{Br:CTL}, Corollary III, p. 282).

 We will prove Theorem~\ref{th:nd} in Section~\ref{sec:proof}. In fact, there we will obtain a more
 general non-dualizability result for nilpotent algebras with a non-abelian, supernilpotent congruence
 (Theorem~\ref{th:sn}).
 In the next section, we will define this notion of supernilpotence (a stronger condition than nilpotence),
 give equivalent formulations of the conditions of the theorem, and prove several auxiliary results on
 Mal'cev algebras.
 In Section~\ref{sec:proof} we will then generalize a construction that Szab{\'o} used on rings
 in~\cite{Sz:FNR} to our setting to prove Theorem~\ref{th:nd}.

In Section \ref{sec:clone}, we  describe 
 a new  clone theoretic approach to duality that was originally suggested by Willard~\cite{Wi:FUPC} and further developed by Davey, Pitkethly, and Willard in~\cite{DP:LAE}.
 Finally, in Section~\ref{sec:dualizable}, we exhibit a non-abelian nilpotent expansion of $\algop{\Z_4}{+}$
 with infinitely many operations that is dualizable. As far as we are aware this is the first example
 of a dualizable algebra that is nilpotent but non-abelian. It shows that it is really supernilpotence
 which prevents dualizability, not nilpotence on its own.
 It also demonstrates that Theorem~\ref{th:nd} does not generalize to algebras of infinite type.
 Our example appears to be the first instance of a dualizable algebra of infinite signature,
 all of whose finitary non-abelian reducts with group operation are non-dualizable.

Note that we will delay defining the notion of dualizability until Section \ref{sec:clone}, when technical
details will become more important. We will instead
 provide a standard lemma giving conditions for non-dualizability.
 We refer the reader to Clark and Davey \cite{CD:NDWA}
for a background on natural duality.

\section{Nilpotence and supernilpotence} \label{sec:nilpotent}

 We denote the set of all $k$-ary term operations on an algebra $\A$ by $\Clo_k(\A)$ and call
 $\Clo(\A) := \bigcup_{k\in\N} \Clo_k(\A)$ the \emph{clone of term operations on} $\A$~\cite[Definition~4.2]{MMT:ALVV}.
 The clone of \emph{polynomial functions} $\Pol(\A)$ on $\A := \algop{A}{F}$ is formed by the fundamental
 operations $F$, the constant functions on $A$, the projections from $A^k$ to $A$ for $k\in\N$ --
 and all compositions thereof~\cite[Definition~4.4]{MMT:ALVV}.

 In~\cite{AM:SAHC} a stronger condition than nilpotence is defined, based on the concept of the higher
 commutators $[\alpha_1, \ldots, \alpha_k]$ as introduced by Bulatov \cite{Bu:NFMA}.

\begin{de} \cite{Bu:NFMA}
Let $k \ge 2$, and $\alpha_1,\ldots,\alpha_k, \nu$ be congruences on an algebra $\A$. We say that
$\alpha_1,\dots, \alpha_{k-1}$ {\em centralize} $\alpha_k$ modulo $\nu$ if for all polynomial operations
 $f(\bar{x}_1,\dots,\bar{x}_k)$ of $\A$ and tuples
 $\bar{a}_1,\ldots,\bar{a}_k$, $\bar{b}_1,\ldots,\bar{b}_k$ over $A$ that satisfy
\begin{enumerate}
\item $\bar{a}_i \equiv_{\alpha_i} \bar{b}_i$ for all $i\in\{1,\dots,k\}$ and
\item $f(\bar{x}_1,\ldots,\bar{x}_{k-1},\bar{a}_k) \equiv_\nu f(\bar{x}_1,\dots,\bar{x}_{k-1},\bar{b}_k)$
 for all
$$(\bar{x}_1,\dots,\bar{x}_{k-1})\in \left(\{\bar{a}_1,\bar{b}_1\}\times \cdots \times\{\bar{a}_{k-1},\bar{b}_{k-1}\}\right)-\{(\bar{b}_1,\dots,\bar{b}_{k-1})\}$$

\end{enumerate}
we also have
 $$f(\bar{b}_1,\dots,\bar{b}_{k-1},\bar{a}_k)\equiv_\nu f(\bar{b}_1,\dots,\bar{b}_{k-1},\bar{b}_k).$$

 We now let the $k$-\emph{ary commutator} $[\alpha_1, \ldots, \alpha_k]$ be the smallest congruence
 $\nu$ on $\A$ such that $\alpha_1, \dots, \alpha_{k-1}$ centralize $\alpha_k$ modulo $\nu$.
\end{de}

 One can show that for a group $\G$ with normal subgroups $N_1,\ldots, N_k$ the $k$-ary commutator corresponds
 to the product of the iterated binary commutators from classical group theory~\cite[Lemma 3.6]{Ma:MASC}.
 For a ring $\Ri$ with ideals $I_1,\ldots, I_k$ the $k$-ary commutator corresponds to the ideal generated
 by the products of $I_1,\ldots, I_k$ in all permutations~\cite[Lemma 3.5]{Ma:MASC}.

 We refer the reader to~\cite{AM:SAHC} and~\cite{Ma:MASC}
 regarding further details of higher commutators. For our results, it is sufficient to know
 that the $k$-ary higher commutator is well-defined and
  coincides with the term condition commutator from~\cite{FM:CTFC} if $k=2$
 (see Lemma~\ref{le:hcgen} below for a description of $[\alpha_1,\dots,\alpha_k]$ specialized to nilpotent
 algebras).

\begin{de} \cite{AM:SAHC}
 Let $k \in \N$, and $\A$  an algebra in a congruence permutable variety. A congruence $\alpha$ on $\A$ is
 \emph{$k$-supernilpotent} if
  $$[\underbrace{\alpha,\dots,\alpha}_{k+1}]=0,$$
 with $0$ the equality relation on $A$.
 The algebra $\A$ is \emph{$k$-supernilpotent} if the total relation $1$ on $A$ is $k$-supernilpotent.
  An algebra or a congruence is \emph{supernilpotent} if it is $k$-supernilpotent for some $k$.
 \end{de}

\begin{de}\cite[p. 179]{Ke:CMVW}, \cite[p. 124]{FM:CTFC}
 Let $A$ be a set, let $k\in\N$.
 Then $c\colon A^{k+1}\to A$ is a \emph{commutator} if $\forall x_1,\dots,x_{k},z\in A$:
\[ z\in\{x_1,\dots,x_{k}\}\Rightarrow c(x_1,\dots,x_k,z) = z. \]
 The commutator $c$ has \emph{rank} $k$ if $\exists a_1,\dots,a_{k},\o\in A\colon c(a_1,\dots,a_k,\o) \neq \o$.
 Otherwise we say it is \emph{trivial}.
\end{de}

 By~\cite[Lemma 7.5]{AM:SAHC} an algebra $\A$ in a congruence permutable variety is $k$-\emph{supernilpotent} if and only if
 $\A$ is nilpotent and all non-trivial commutators in $\Pol(\A)$ have rank at most $k$
 (see also Lemma~\ref{le:hcgen} below).

 Let $\A$ be a finite nilpotent algebra of finite type that generates a congruence modular variety.
We then have that the following properties are equivalent:
\begin{enumerate}
\item \label{it:1}
 $\A$ is a direct product of algebras of prime power order.
\item \label{it:2}
 $\exists M$ such that every non-trivial commutator in $\Clo(\A)$ has rank at most $M$.
\item \label{it:3}
 $\A$ is supernilpotent.
\end{enumerate}

 \eqref{it:1}$\Rightarrow$\eqref{it:2} follows from~\cite[Theorem 14.16]{FM:CTFC}.
 \eqref{it:2}$\Rightarrow$\eqref{it:1} is due to Kearnes~\cite[Theorem 3.14]{Ke:CMVW}.
 \eqref{it:1}$\Leftrightarrow$\eqref{it:3} was proved by Aichinger and Mudrinski~\cite[Lemma 7.6]{AM:SAHC}.

\begin{lem}  \label{le:c}
 Let $\A$ be an algebra with Mal'cev term operation $m$, let $\alpha$ be a central congruence of $\A$,
 and let $a,b,c,\o\in A$ with $(c,\o)\in\alpha$. Then
\[ m(m(a,\o,b),\o,c) = m(a,\o,m(b,\o,c)) = m( m(a,\o,c),\o,b) \]
 and
\[ m(a,c,\o) = m(a,\o,m(\o,c,\o)). \]
\end{lem}

\begin{proof}
 As $[\alpha,1]=0$, \cite[Proposition 5.7]{FM:CTFC} implies that
$$\begin{array}{cl}
&m(m(a_1,a_2,a_3), m(b_1,b_2,b_3), m(c_1,c_2,c_3))\\
=& m(m(a_1,b_1,c_1), m(a_2,b_2,c_2), m(a_3,b_3,c_3))
\end{array}$$
whenever $b_i \equiv_\alpha c_i$ for all $i\in\{1,2,3\}$. Hence
$$\begin{array}{rcl} %
m(m(a,o,b),o,c) & = & m(m(a,o,b),m(o,o,o),m(o,o,c)) \\
&=&m(m(a,o,o),m(o,o,o),m(b,o,c))\\
&=&m(a,o,m(b,o,c))\end{array}$$
$$\begin{array}{rcl} %
m(m(a,o,b),o,c) & = & m(m(a,o,b),m(o,o,o),m(c,o,o)) \\
&=&m(m(a,o,c),m(o,o,o),m(b,o,o))\\
&=&m(m(a,o,c),o,b)\end{array}$$
and
$$\begin{array}{rcl} %
m(a,c,\o) & = & m(m(a,\o,\o),m(c,c,c),m(c,c,\o)) \\
&=&m(m(a,c,c),m(\o,c,c),m(\o,c,\o))\\
&=&m(a,\o,m(\o,c,o))\end{array}.$$
\end{proof}

\begin{lem} \label{le:translate} \cite[Lemma 7.3, Corollary 7.4]{FM:CTFC}
 Let $\A$ be a nilpotent algebra with Mal'cev term operation $m$.
 Then there exists $f\in\Clo_3(\A)$ such that $\forall x,b,c\in\A\colon$
\[ m( f(x,b,c), b,c ) = x \text{ and } f( m(x,b,c), b,c ) = x. \]
 In particular, for all $b,c\in A$
\[ t\colon A\to A, x\mapsto m(x,b,c) \]
 is a bijection.
\end{lem}

 In the proof of Theorem~\ref{th:nd} we will need a more explicit version of Lemma 14.6 from~\cite{FM:CTFC}
 describing term operations on nilpotent algebras which we will state next.

 Set $\kset := \{1,\ldots,k\}$, and for $x\in A^k$ and $S\subseteq \kset $, let
\[ x_S := (x_i)_{i\in S}. \]

\begin{lem} \cite[cf. Lemma 14.6]{FM:CTFC} \label{le:csum}
 Let $\A$ be a finite nilpotent algebra with Mal'cev term operation $m$,
 let $k\in\N$, and fix a linear order $\preceq$ on the power set of $\kset$.

 Let $f\in\Clo_{k+1}(\A)$. Then for every $S\subseteq \kset,S\neq\emptyset,$ there exists a commutator
 $c_S\in\Clo_{|S|+1}(\A)$ such that for all $x\in A^k,z\in A$
\[ f(x,z) = f(z,\dots,z,z) +_z \sum_{S\subseteq \kset,S\neq\emptyset} c_S(x_S,z) \]
 where the sum is taken with respect to $a+_zb := m(a,z,b)$, associated left to right and ordered with
 respect to $\preceq$.
\end{lem}

\begin{proof}
 In the following, all additions refer to $+_z$,
 and all sums are associated and ordered as in the statement of the lemma.
 Let $1$ denote the total relation on $A$, let $0$ denote the equality relation on $A$.
 Let $(1,1]^0 := 1$ and $(1,1]^{n+1} := [1,(1,1]^n]$ for $n\in\N_0$.

 First we fix $n\in\N$ and consider an operation $e\in\Clo_{k+1}(\A)$ such that $\forall x\in A^k,z\in A$
\[ e(x,z) \equiv z \mod (1,1]^n \text{ and } e(z,\dots,z,z) = z. \]
 We claim that there exist commutators $d_S\in\Clo_{|S|+1}(\A)$, $S\subseteq \kset,S\neq\emptyset$,
 such that $\forall x\in A^k,z\in A$
\begin{equation} \label{eq:es}
 e(x,z) \equiv \sum_{S\subseteq \kset,S\neq\emptyset} d_S(x_S,z) \mod (1,1]^{n+1} \text{ and } d_S(x_S,z) \equiv z \mod (1,1]^n.
\end{equation}
 We prove this by induction on $k$.
 If $k = 1$, then $e$ itself is a commutator, and the assertion is trivially true.
 Assume $k > 1$ in the following.
 Let $e_o := e$, and for $j\in\kset$ define $e_j\in\Clo_{k+1}(\A)$ iteratively by
\[ \begin{array}{rl} &\!\!\!\!\!\!\!\!\!\!\!\!\!\!\!\!\!\!\!\!\!\!\!\!\!\!\!\!\!\! e_j(x_1,\dots,x_k,z)\\
 \quad\quad := &\!\!m( e_{j-1}(x_1,\dots,x_k,z),e_{j-1}(x_1,\dots,x_{j-1},z,x_{j+1},\dots,x_k,z),z). \end{array}\]
 Then $e_j(x,z) \equiv z \mod (1,1]^n$ for all $x\in A^k,z\in A$ and $e_j(x_1,\dots,x_k,z) = z$ whenever
 $x_l = z$ for some $l\leq j$. In particular, $e_{k}$ is a commutator.

 For $T\subseteq\kset$ define $\delta_T\colon A^{k+1}\to A^{k+1}$ by
\[ (\delta_T(x_1,\dots,x_k,z))_i := \begin{cases} x_i & \text{if } i\in T, \\
                                  z & \text{else}. \end{cases} \]
 For $a,z\in A$ write $-_z a := m(z,a,z)$. Note that $(1,1]^n/(1,1]^{n+1}$ is central in $\A/(1,1]^{n+1}$.
 Hence, for every $z\in A$, the operations $+_z,-_z$, and $z$ induce addition, inverse, and zero element
 of an abelian group on $\{a/(1,1]^{n+1}\setsuchthat a\equiv z \mod (1,1]^n\}$ by Lemma~\ref{le:c}.
 From the definitions and Lemma~\ref{le:c} it is straightforward that $\forall x\in A^k,z\in A$
\[ e_{k}(x,z) \equiv \sum_{T\subseteq\kset} (-1)^{k-|T|} e\left(\delta_T(x,z)\right) \mod (1,1]^{n+1} \]
 with $-$ referring to $-_z$ and the order of the sum irrelevant because the summands commute ($\!\!\!\!\mod$ $ (1,1]^{n+1})$ with
 respect to $+_z$.
 Since $e\circ\delta_{\kset} = e$, we obtain $\forall x\in A^k,z\in A$
\begin{equation} \label{eq:es2}
 e(x,z) \equiv e_{k}(x,z) + \sum_{T\subset\kset} (-1)^{k+1-|T|} e\left(\delta_T(x,z)\right) \mod (1,1]^{n+1}.
\end{equation}
 Let $T\subset\kset$. Then $e\circ\delta_T$ does not depend on $x_i$ for $i\in\kset\setminus T$.
 From the induction assumption for $k-1$ it follows that we have commutators
 $d^{T}_S\in\Clo_{|S|+1}(\A)$, for $S\subseteq T,S\neq\emptyset$, such that $\forall x\in A^k,z\in A$
\begin{equation*}
 e\left(\delta_T(x,z)\right) \equiv \sum_{S\subseteq T,S\neq\emptyset} d^{T}_S(x_S,z) \!\mod (1,1]^{n+1} \text{, } d^{T}_S(x_S,z) \equiv z \!\mod (1,1]^n.
\end{equation*}
 Hence for every $x\in A^k, z\in A$ we obtain
\begin{eqnarray*} 
 & &\sum_{T\subset\kset} (-1)^{k+1-|T|} e\left(\delta_T(x,z)\right)  \\& \equiv & \sum_{T\subset\kset} (-1)^{k+1-|T|} \sum_{S\subseteq T,S\neq\emptyset} d^{T}_S(x_S,z) \mod (1,1]^{n+1} \\
 & \equiv & \sum_{S\subset\kset,S\neq\emptyset} \underbrace{\sum_{S\subseteq T\subset \kset}  (-1)^{k+1-|T|} d^{T}_S(x_S,z)}_{=: d_S(x_S,z)} \mod (1,1]^{n+1}.
\end{eqnarray*}
 We observe that $d_S\in\Clo_{|S|}(\A)$ is a commutator because $d^{T}_S$ is a commutator for every
 $S\subseteq T\subset\kset$ and that $d_S(x_S,z)\equiv z \mod (1,1]^n$ for every $x\in A^k, z\in A$.
 Finally~\eqref{eq:es2} yields
\[ e(x,z) \equiv \underbrace{e_k(x,z)}_{=: d_{\kset}(x,z)} +\sum_{S\subset\kset,S\neq\emptyset} d_S(x_S,z) \mod (1,1]^{n+1} \]
 which proves~\eqref{eq:es}.

 Next we consider an arbitrary operation $f\in\Clo_{k+1}(\A)$. Let $n\in\N_0$.
 We claim that we have commutators $c_S\in\Clo_{|S|+1}(\A)$ for $S\subseteq \kset,S\neq\emptyset$,
 such that $\forall x\in A^k,z\in A$
\begin{equation} \label{eq:fmod1n}
 f(x,z) \equiv f(z,\dots,z,z) + \sum_{S\subseteq \kset,S\neq\emptyset} c_S(x_S,z) \mod (1,1]^n.
\end{equation}
 For the proof we use induction on $n$. For $n = 0$ the statement is trivially true.
 So assume we have~\eqref{eq:fmod1n} for some fixed $n \ge 0$.
 By Lemma~\ref{le:translate} there exists $e\in\Clo_{k+1}(\A)$ such that $\forall x\in A^k,z\in A$
\begin{equation} \label{it:e1}
 f(x,z) = f(z,\dots,z,z) + \sum_{S\subseteq \kset,S\neq\emptyset} c_S(x_S,z) + e(x,z)
\end{equation}
 as well as $e(x,z) \equiv z \mod (1,1]^n$ and $e(z,\dots,z,z) = z$.
 Now $e$ can be written as a sum of commutators $d_S$ modulo $(1,1]^{n+1}$ as in~\eqref{eq:es}.
 Then~\eqref{eq:es} and~\eqref{it:e1} yield
\[ f(x,z) \equiv f(z,\dots,z,z) + \!\sum_{S\subseteq \kset,S\neq\emptyset} \! c_S(x_S,z) + \!\sum_{S\subseteq \kset,S\neq\emptyset} \! d_S(x_S,z) \mod (1,1]^{n+1}. \]
 From Lemma~\ref{le:c} we obtain
\[ f(x,z) \equiv f(z,\dots,z,z) + \sum_{S\subseteq \kset,S\neq\emptyset} (c_S(x_S,z) + d_S(x_S,z)) \mod (1,1]^{n+1}. \]
 Since $c_S+d_S$ is a commutator for every $S$, the induction step is proved.
 Thus we can represent $f$ as in~\eqref{eq:fmod1n} for every $n\in\N_0$.
 For $N$ such that $(1,1]^N = 0$ this yields the lemma.
\end{proof}

 We conclude the section with three observations on commutator polynomials and higher commutators of congruences.

\begin{lem} \label{le:c2}
 Let $\A$ be an algebra in a congruence modular variety, let $k\geq 2$, and let $c\in\Clo_{k+1}(\A)$
 be a commutator. Let $\alpha,\beta$ be congruences of $\A$, and let $a_1,\dots,a_k,\o\in A$ such that
 $a_1\equiv_\alpha \o$, $a_2\equiv_\beta \o$. Then
\[ c(a_1,\dots,a_k,\o) \equiv_{[\alpha,\beta]} \o. \]
\end{lem}

\begin{proof}
 Consider $M_\A(\alpha,\beta)$, the subuniverse of $\A^{2\times 2}$ that is generated by all the elements
\[ \mtt{a}{b}{a}{b}, \mtt{e}{e}{f}{f} \text{ for } a,b,e,f\in A \text{ with } a\equiv_\alpha b, e\equiv_\beta f. \]
 Then
\[ c\left( \mtt{\o}{a_1}{\o}{a_1},\mtt{\o}{\o}{a_2}{a_2}, \mtt{a_3}{a_3}{a_3}{a_3}, \dots, \mtt{a_k}{a_k}{a_k}{a_k},\mtt{\o}{\o}{\o}{\o}\right) \]
\[= \mtt{\o}{\o}{\o}{c(a_1,\dots,a_k,\o)} \]
 is contained in $M_\A(\alpha,\beta)$ and the result follows from the definition of the term condition commutator.
\end{proof}

\begin{lem} \label{le:hcgen}
 Let $\A$ be a nilpotent algebra generating a congruence modular variety, let $k\in\N$, and
 let $\alpha_1,\dots,\alpha_k$ be congruences of $\A$.
 Then $[\alpha_1,\dots,\alpha_k]$ is the congruence of $\A$ that is generated by
\begin{eqnarray*}
 T := \{ (c(a_1,\dots,a_k,\o), \o) & : & c\in\Pol_{k+1}(\A), c \text{ is a commutator}, \\
 & & (a_1,\o)\in\alpha_1,\dots,(a_k,\o)\in\alpha_k \}.
\end{eqnarray*}
\end{lem}

\begin{proof}
 Let $(\o_1,\dots,\o_k)\in A^k$. As in Definition 4.9 of~\cite{AM:SAHC} we say that $f\colon A^k\to A$ is
 absorbing at $(\o_1,\dots,\o_k)$ if for all $(x_1,\dots,x_k)\in A^k\colon$
\[ (\exists i\in\kset\colon x_i = \o_i) \Rightarrow f(x_1,\dots,x_k) = f(\o_1,\dots,\o_k). \]
 By~\cite[Lemma 6.9]{AM:SAHC} $[\alpha_1,\dots,\alpha_k]$ is the congruence of $\A$ that is generated by
\begin{eqnarray*}
 R := \{ (f(b_1,\dots,b_k), f(\o_1,\dots,\o_k)) & : & f\in\Pol_{k}(\A), f \text{ is absorbing at } (\o_1,\dots,\o_k), \\
 & & (b_1,\o_1)\in\alpha_1,\dots,(b_k,\o_k)\in\alpha_k \}.
\end{eqnarray*}
 Hence it suffices to prove
\begin{equation} \label{eq:TR}
 T = R.
\end{equation}
 For the inclusion $\subseteq$ let $(a,\o)\in T$.
 Then we have a commutator $c$ in $\Pol_{k+1}(\A)$, and $(a_1,\o)\in\alpha_1,\dots,(a_k,\o)\in\alpha_k$
 such that $(c(a_1,\dots,a_k,\o), \o) = (a,\o)$.

 Consider $f\colon A^k\to A, (x_1,\dots,x_k) \mapsto c(x_1,\dots,x_k,\o)$. Then $f$ is a $k$-ary polynomial
 function on $\A$ that is absorbing at $(\o,\dots,\o)$. Since $(f(a_1,\dots,a_k),f(\o,\dots,\o)) = (a,\o)$,
 we obtain $(a,\o)\in R$.

 For the converse inclusion $\supseteq$ in~\eqref{eq:TR}, let $(a,\o)\in R$.
 Then we have $(b_1,\o_1)\in\alpha_1,\dots,(b_k,\o_k)\in\alpha_k$ and some $f\in\Pol_{k}(\A)$ that is absorbing at
 $(\o_1,\dots,\o_k)$ such that $(f(b_1,\dots,b_k), f(\o_1,\dots,\o_k)) = (a,\o)$.
 Since $\A$ is nilpotent, it has a Mal'cev term operation $m$ by~\cite[Theorem 6.2]{FM:CTFC}.
 Define a $(k+1)$-ary polynomial operation $c$ on $\A$ by
\[ c(x_1,\dots,x_k,z) := m(f(m(x_1,z,\o_1),\dots,m(x_k,z,\o_k)),\o,z). \]
Note that $c$ is a commutator.
 Let $i\in\kset$. By Lemma~\ref{le:translate}, we have $a_i\in A$ such that $m(a_i,\o,\o_i) = b_i$.
 Note that $\o_i \equiv_{\alpha_i} b_i$ implies $a_i \equiv_{\alpha_i} \o$.
 Now $c(a_1,\dots,a_k,\o) = f(b_1,\dots,b_k)$ and $(a,\o)\in T$.
\end{proof}

 Let $m$ be a Mal'cev term operation on an algebra $\A$.
 Under certain conditions commutator terms on $\A$ turn out to be multilinear and
 alternating with respect to the operation $a+_zb := m(a,z,b)$.

\begin{lem}  \label{le:ma}
 Let $\A$ be a nilpotent algebra with Mal'cev term operation $m$, let $k\in\N$, let $c\in\Clo_{k+1}(\A)$
 be a commutator, and let $\alpha\in\Con(\A)$ be $k$-supernilpotent.
\begin{enumerate}
\item  \label{it:multi}
 Then $\forall x_1,\dots,x_k,y,z\in A$ with $(x_1,z),\dots,(x_k,z),(y,z)\in\alpha\colon$
\[ c(x_1,\dots,x_i+_z y,\dots,x_k,z) = c(x_1,\dots,x_i,\dots,x_k,z)+_z
c(x_1,\dots,\underset{i}{y},\dots,x_k,z) \]
\vspace{-10pt}

\hspace*{\fill}$\text{ (multilinearity)}$.
\item  \label{it:alt}
 Let $i,j\in\{1,\dots,k\}$ be distinct. Assume that
 $\forall x_1,\dots,x_k,z\in A$ with $(x_1,z),\dots,(x_k,z)\in\alpha\colon$
\[ x_i = x_j \Rightarrow c(x_1,\dots,x_k,z) = z. \]
 Then $\forall x_1,\dots,x_k,z\in A$ with $(x_1,z),\dots,(x_k,z)\in\alpha\colon$
\[ c(x_1,\dots,x_i,\dots, x_j,\dots,x_k,z) +_z c(x_1,\dots,\underset{i}{x_j},\dots,\underset{j}{x_i},\dots,x_k,z) = z \]
\vspace{-10pt}

\hspace*{\fill}$\text{ (alternating)}.$
\end{enumerate}
\end{lem}

\begin{proof}
For simplicity, we will write $+$ instead of $+_z$ throughout the proof.
 For~\eqref{it:multi} define \vspace{3pt}

\noindent $d(x_1,\dots,x_k,y,z)$ \vspace{-3pt}
\[:=  m\left(c(x_1,\dots,x_i+ y,\dots,x_k,z),c(x_1,\dots,x_i,\dots,x_k,z)+ c(x_1,\dots,\underset{i}{y},\dots,x_k,z),z\right). \]
 Then $d$ is a commutator.
 Let $x_1,\dots,x_k,y,z\in A$ be fixed such that $(x_1,z),\dots,$ $(x_{k},z), (y,z)\in\alpha$.
Since
$[\underbrace{\alpha,\dots,\alpha}_{k+1}] = 0$, Lemma~\ref{le:hcgen} yields
 $d(x_1,\dots,x_k,y,z) = z$.
 Since $$m\left(c(\dots,x_i,\dots)+ c(\dots,\underset{i}{y},\dots),c(\dots,x_i,\dots)+c(\dots,\underset{i}{y},\dots),z\right) = z,$$
 Lemma~\ref{le:translate} yields $c(\dots,x_i+y,\dots) = c(\dots,x_i,\dots)+c(\dots,y,\dots)$.

 For~\eqref{it:alt}
 we note that~\eqref{it:multi} implies
\begin{eqnarray*}
 c(\dots,x_i+x_j,\dots, x_i+x_j,\dots) & = & c(\dots,x_i,\dots, x_i,\dots)+c(\dots,x_i,\dots, x_j,\dots) \\
                                      & & +c(\dots,x_j,\dots, x_i,\dots)+c(\dots,x_j,\dots, x_j,\dots).
\end{eqnarray*}
\end{proof}

\section{Proof of Theorem~\ref{th:nd}} \label{sec:proof}

 This section consists almost entirely of the proof of the following result which will then yield
 Theorem~\ref{th:nd} in the end.

\begin{thm}  \label{th:sn}
 Let $\A$ be a finite nilpotent algebra in a congruence modular variety.
 Assume that $\A$ has a non-abelian $k$-supernilpotent congruence $\alpha$ (such that
 $[\underbrace{\alpha,\dots,\alpha}_{k+1}] = 0$) for some $k\geq 2$.
 Then $\A$ is inherently non-dualizable, i.e. every finite superalgebra $\B$ of $\A$ is
non-dualizable.
\end{thm}

 Let $\A$ and $\B$ satisfy the assumptions of Theorem~\ref{th:sn}. 
 We will show that $\B$ is not dualizable using the ghost element method in the form of the following lemma.

\begin{lem}[Non-dualizability {\cite[Lemma 5.2]{CD:CDT}}]  \label{le:nd}
Let $\B$ be a finite algebra and let $\nn \in \mathbb N$. Assume there is a
subalgebra $\D$ of\/~$\B^J$, for some set\/~$J$, and an infinite subset\/ $D_0$
of\/ $D$ such that
\begin{enumerate}
 \item for each homomorphism $\varphi \colon  \D \to \B$, the equivalence
     relation $\ker (\varphi|_{D_0})$ has a unique block of size more
     than\/~$\nn$, and
 \item the algebra $\D$ does not contain the element\/ $g$
     of\/ $B^J$ given by $g(i) := b_i(i)$, where $i\in J$ and $b_i$ is
     any element of the unique infinite block of\/ $\ker
     (\pi_i|_{D_0})$.
\end{enumerate}
Then $\B$ is non-dualizable.
\end{lem}

 Since $\A$ is nilpotent, it has a Mal'cev term operation $m$ by~\cite[Theorem 6.2]{FM:CTFC}.
 By assumption there exists a non-abelian congruence $\beta$ of $\A$ and some $k' \geq 2$ such that
 $[\underbrace{\beta,\dots,\beta}_{k'+1}] = 0$.
 Let $\alpha\in\Con(\A)$ be minimal such that $\alpha\leq\beta$ and $\alpha$ is non-abelian.
 Since $\A$ is nilpotent, $\gamma := [\alpha,\alpha]$ is strictly smaller than $\alpha$ and
 consequently abelian.
 Further we have $2\leq k\leq k'$ such that
 $[\underbrace{\alpha,\dots,\alpha}_{k+1}] = 0$ but $[\underbrace{\alpha,\dots,\alpha}_k] \neq 0$.
 Hence, by Lemma~\ref{le:hcgen}, there exist a commutator $c\in\Pol_{k+1}(\A)$ and
 $(x_1,z),\dots,(x_k,z)\in\alpha$ such that $c(x_1,\dots,x_k,z) \neq z$.
 We will distinguish the following two cases.

\emph{Case 1, all commutators $c\in\Pol(\A)$ of rank $k$ satisfy $c(x_1,\dots,x_k,z) = z$ whenever $x_i = x_j$
 for some $i\neq j$ and $(x_1,z),\dots,(x_k,z)\in\alpha$:}
 Fix a commutator $f\in\Pol_{k+1}(\A)$, and fix $\o\in A, (a_1,o), \dots,$ $(a_k,\o)\in\alpha$ such that
 $f(a_1,\dots,a_k,\o) \neq\o$.

\emph{Case 2, there exists a commutator $c\in\Pol(\A)$ of rank $k$ and $(x_1,z),\dots,(x_k,z)\in\alpha$
 such that $x_i = x_j$ for some $i\neq j$ and $c(x_1,\dots,x_k,z) \neq z$:}
 By permuting coordinates if necessary, we then also have a commutator $f\in\Pol_{k+1}(\A)$ of rank $k$
 and $\o\in A, (a_1,o), \dots,(a_k,\o)\in\alpha$ with $a_1 = a_2$ such that $f(a_1,\dots,a_k,\o) \neq \o$.

 With elements $a_1,\dots,a_k,\o$ and commutator $f$ chosen according to the above cases we proceed to construct
 the subuniverse $D$ of $(\B^{\P(\kset)})^\Z$ (where $\P(\kset)$ is the power set of $\kset$).

 For $a\in A$, let $\aa\in A^{\P(\kset)}$ such that $\aa(S) := a$ for all $S\subseteq \kset$.
 For $i\in\kset$, define $u_i\in A^{\P(\kset)}$ by
\[ u_i(S) := \left\{ \begin{array}{ll} a_i & \text{if } i\in S, \\
                                      \o & \text{else.} \end{array}\right. \]
 Let $t := 2|B|+1$. For $i\in\Z$ let $d_i\in (A^{\P(\kset)})^\Z$ be given by
 \[ d_i(j) := \left\{ \begin{array}{ll} u_1 & \text{if } j \in \{i, i+t+3\}, \\
                                      u_2 & \text{if } j \in \{i +1, i+t+2\}, \\
                                      \oo & \text{else.} \end{array}\right. \]
 Then
\[ d_i = ( \dots, \oo, \underset{i}{u_1}, u_2,\underbrace{\oo, \dots,\oo}_{t}, u_2, u_1,\oo,\dots ). \]
 For $l\in\{3,\dots,k\}$, let
\[ c_l := ( \dots,u_l, u_l,\dots) \in (A^{\P(\kset)})^\Z, \]
 and for $a\in A$ let
\[ \aaa := ( \dots, \aa,\aa, \dots) \in (A^{\P(\kset)})^\Z, \]
i.e., $c_l(j)=u_l$ and $\aaa(j)=\aa$ for all $j \in \Z$.

 We let $D$ be the subuniverse of $(\B^{\P(\kset)})^\Z$ that is generated by
 $\{d_i\setsuchthat i\in\Z\}\cup\{c_3,\dots,c_k\}\cup\{\bar{\bar{a}}\setsuchthat a\in A\}$.
 Note that $D\subseteq (A^{\P(\kset)})^\Z$.

 To describe the set $D_0$ from Lemma~\ref{le:nd} we construct yet more elements in $D$. For $i\in\Z$ consider
\[ \begin{array}{cccccccccccc}
 d_i & = & (\dots  \oo & \overset{i-t-2}{\oo} & \oo & \oo  \dots  \oo& \overset{i}{u_1} & u_2  & \oo \dots \oo& \overset{i+t+2}{u_2} & u_1 & \oo \dots ) \\
 d_{i-t-2} & = & ( \dots \oo & u_1 & u_2 & \oo  \dots  \oo & u_2 & u_1 &\oo  \dots \oo & \oo &\oo & \oo \dots )
\end{array} \]
 and define $v_{i,i+1} := f(d_i, d_{i-t-2},c_3,\dots,c_k,\ooo)$. Then
\[ v_{i,i+1}(j) = \left\{ \begin{array}{ll}
    f(u_1,u_2,u_3,\dots,u_k,\oo) & \text{if } j = i, \\
    f(u_2,u_1,u_3,\dots,u_k,\oo) & \text{if } j = i+1, \\
    \oo & \text{else}. \end{array}\right. \]
 Now let $e := f(u_1,\dots,u_k,\oo)$. Since $f$ is a commutator, we have
\[ e(S) = \left\{ \begin{array}{ll} f(a_1,\dots,a_k,\o) & \text{if } S = \kset, \\
                                                                   \o & \text{else}. \end{array}\right. \]
 We recall that $\gamma = [\alpha,\alpha]$ is abelian. For $x,y\in\o/\gamma$ define
\[ x+_\o y := m(x,\o,y). \]
 Then $\algop{\o/\gamma}{+_\o}$ is an abelian group by Lemma~\ref{le:c}.

 Since $k\geq 2$, we have that $f(a_1,\dots,a_k,\o) \equiv_\gamma \o$ by Lemma~\ref{le:c2}.
 Hence all entries of $f(u_1,u_2,u_3,\dots,$ $u_k,\oo)$ are contained in $\o/\gamma$.
 Similarly $f(u_2,u_1,u_3,\dots,u_k,\oo) \in (\o/\gamma)^{\P(\kset)}$.

\emph{In Case 1:}
 Lemma~\ref{le:ma} yields $f(a_2,a_1,a_3,\dots,a_k,\o) = -f(a_1,a_2,a_3,\dots,a_k,\o)$ in the group
 $\algop{\o/\gamma}{+_\o}$. Hence
\[ v_{i,i+1} = ( \dots,\oo, \underset{i}{e}, \underset{i+1}{-e},\oo,\dots ). \]
 For $i < j$ define $v_{i,j} := \sum_{l=i}^{j-1}v_{l,l+1}$. Then
\[ v_{i,j} = ( \dots,\oo, \underset{i}{e},\oo,\dots,\oo, \underset{j}{-e},\oo,\dots ). \]
\emph{In Case 2:} From $a_1 = a_2$ we obtain
\[ v_{i,i+1} = ( \dots,\oo, \underset{i}{e}, \underset{i+1}{e},\oo,\dots ). \]
 For $i<j$ let $v_{i,j} := \sum_{l=i}^{j-1}(-1)^{l-i}v_{l,l+1}$, which yields
\[ v_{i,j} = ( \dots,\oo, \underset{i}{e},\oo,\dots,\oo, \underset{j}{(-1)^{j-i-1} e},\oo,\dots ). \]

 We want to use Lemma~\ref{le:nd} with
 $D_0:= \{v_{0,i} \setsuchthat i\in\N\}$, $\nn:=2|B|(2|B|-1)$, and the ghost element $g$ with $g(0):=e$
 and $g(i):=\oo$ for $i \ne 0$.
 We will first establish the second condition of the Lemma, i.e., that the ghost is not in $D$.
 This still requires some more preparation.

 First we observe that if an element in $D$ is congruent to $\ooo$ modulo $\gamma$, then it can be written
 as sum of a constant and commutators evaluated at generators of $\D$ such that every summand is congruent to
 $\ooo$ modulo $\gamma$.

\begin{lem} \label{le:local}
 Let $l < r$, $n := r-l+k-1$, and let $h\in\Pol_{n+1}(\A)$ such that
\[ h(d_l,d_{l+1},\dots,d_r,c_3,\dots,c_k,\ooo) \equiv_\gamma \ooo. \]
 For $S\subseteq \n, S\neq\emptyset$, let $c_S\in\Pol_{|S|+1}(\A)$ be commutators
 such that $\forall x\in A^{n},z\in A$
\[ h(x,z) = h(z,\dots,z,z) +_z \sum_{S\subseteq \n,S\neq\emptyset} c_S(x_S,z). \]
 Write $g_1 := d_l, g_2 := d_{l+1},\dots, g_{r-l+1} :=  d_r,  g_{r-l+2} := c_3, \dots, g_{n} :=c_k$.
 Then $h(\ooo,\dots,\ooo) \equiv_\gamma \ooo$ and
 $\forall S \subseteq \n, S\neq\emptyset$:
\[ c_S(g_S,\ooo) \equiv_\gamma \ooo. \]
\end{lem}

\begin{proof}
 We will once again write $+$ instead of $+_z$. We have
\[ \begin{array}{ccccccccccccccc}
    g_1 & = & (\dots & \overset{l}{u_1} & u_2 &  \oo   &    \dots    & u_2 & \overset{l+t+3}{u_1}   & \oo    &\dots   & \oo & \oo    &      & \dots )\\
    g_2 & = & (\dots &  \oo   & u_1 & u_2 & \dots  & \oo    & u_2   & u_1 & \dots  & \oo  & \oo       &   & \dots )\\
     & \vdots &      &     &    &     &        &     &       &     &   &  &        &   & \\
 g_{t+4} & = & (\dots &   \oo  & \oo   & \oo    & \dots       & \oo    & u_1  & u_2 &   \dots   & u_2 & u_1 &  & \dots ) \\
     & \vdots &      &     &    &     &        &     &       &     &   &  &     &     & \\
 \end{array} \]
 Note that $(h(g_1,\dots, g_n,\ooo))(l)(\emptyset) = h(\o,\dots,\o)$ implies
 $h(\ooo,\dots,\ooo) \equiv_\gamma \ooo$.
 If $|S| \geq 2$, then $c_S(g_S,\ooo) \equiv_{\gamma} \ooo$ by Lemma~\ref{le:c2}.
 So it only remains to prove the assertion for $c_S$ with $|S| = 1$. 
 First we claim
\begin{equation} \label{eq:csa1}
 c_{\{s\}}(a_1,o) \equiv_\gamma \o \text{ for all } s\in\{1,\dots,r-l+1\}.
\end{equation}
 For the proof we use induction on $s$. For $s\in\{1,\dots,t+3\}$ we have
\[ (h(g_1,\dots,g_n,\ooo))(l-1+s)(\{1\}) = h(\o,\dots,\o)+c_{\{s\}}(a_1,o) \]
 and consequently $c_{\{s\}}(a_1,o) \equiv_\gamma \o$.
 Now let $s\in\{t+4,\dots,r-l+1\}$. Then
\[ (h(g_1,\dots,g_n,\ooo))(l-1+s)(\{1\}) = h(\o,\dots,\o)+c_{\{s-(t+3)\}}(a_1,o)+ c_{\{s\}}(a_1,o). \]
 Since $c_{\{s-(t+3)\}}(a_1,o) \equiv_\gamma \o$ by induction assumption, we get $c_{\{s\}}(a_1,o) \equiv_\gamma \o$
 as well. Hence~\eqref{eq:csa1} is proved.

 Similarly we obtain $c_{\{s\}}(a_2,o) \equiv_\gamma \o$, which together with~\eqref{eq:csa1} implies
 $c_{\{s\}}(g_s,\ooo) \equiv_\gamma \ooo$ for all $s\in\{1,\dots,r-l+1\}$.

 Now let $j\in\{3,\dots,k\}$. Then
\[ (h(g_1,\dots,g_n,\ooo))(l)(\{j\}) = h(\o,\dots,\o)+c_{\{r-l+j-1\}}(a_j,o), \]
 which yields $c_{\{s\}}(g_s,\ooo) \equiv_\gamma \ooo$ for all $s\in\{r-l+2,\dots,n\}$.
\end{proof}

 We are now ready to characterize those elements in $D$ that are congruent to $\ooo$ modulo
 $\gamma$ by their parities.

\begin{lem} \label{le:parity}
 Let $w\in D$ be such that $w(i)(S) \equiv_\gamma \o$ for all $i\in\Z, S\in\P(\kset)$.
 Assume that $w(i) = \oo$ for all but finitely many $i\in\Z$.
Then the following parity conditions hold in the abelian group $\algop{\o/\gamma}{+_\o}$.

 \noindent In Case 1:
\[ \sum_{i\in\Z} \sum_{S\subseteq \kset} (-1)^{|S|} w(i)(S) = \o. \]
 In Case 2:
\[ \sum_{i\in\Z} (-1)^i \sum_{S\subseteq \kset} (-1)^{|S|} w(i)(S) = \o \]
\end{lem}

\begin{proof}
 By Lemmas~\ref{le:csum} and~\ref{le:local}, $w$ is a sum of elements $c(g_1,\dots,g_n,\ooo)$ where
 $c\in\Pol_n(\A)$ is a commutator and
 $\{ g_1,\dots,g_n \} \subseteq \{d_i\setsuchthat i\in\Z\}\cup\{c_3,\dots,c_k\}$.
 Since $[\underbrace{\alpha,\dots,\alpha}_{k+1}] = 0$, we may assume $n \leq k$ (otherwise
 $c(g_1,\dots,g_n,\ooo) = \ooo$ by Lemma~\ref{le:hcgen}). Further, by Lemma~\ref{le:local}, all these summands are congruent to $\ooo$ modulo $\gamma$.
 Since $+_\o$ is commutative on $\o/\gamma$, it suffices to prove the assertion for every summand.
 So assume
 $w = c(g_1,\dots,g_n,\ooo)$ where $c\in\Pol_n(\A)$ is a commutator and
 $\{ g_1,\dots,g_n \} \subseteq \{d_i\setsuchthat i\in\Z\}\cup\{c_3,\dots,c_k\}$.

\emph{Case, $\{g_1(i), \ldots,g_n(i),\oo\}\neq\{u_1,\ldots,u_k,\oo\}$ for all $i\in\Z$:}
 We claim that
\begin{equation} \label{eq:local}
\forall i\in\Z\colon \sum_{S\subseteq \kset} (-1)^{|S|} w(i)(S) = \o.
\end{equation}
 For $i \in\Z$ fixed, let $l\in\kset$ be such that $u_l \not\in \{g_1(i), \ldots,g_n(i)\}$.
 Let $S\subseteq \kset\setminus\{l\}$. Since $u_j(S) = u_j(S\cup\{l\})$ for all $j\in\kset, j\neq l$, we obtain
\[ w(i)(S) = (c(g_1,\dots,g_n,\o))(i)(S) = (c(g_1,\dots,g_n,\o))(i)(S\cup\{l\}) = w(i)(S\cup\{l\}). \]
 From this, ~\eqref{eq:local} and the result follows.

\emph{Case, we have $i\in\Z$ such that $\{g_1(i), \ldots,g_n(i)\}=\{u_1,\ldots,u_k\}$:}
 Then $n = k$. Without loss of generality we may assume that $g_j(i) = u_j$ for all $j\in\kset$. Then
\begin{eqnarray*}
 g_1 & \in & \{d_i,d_{i-t-3} \}, \\
 g_2 & \in & \{d_{i-1},d_{i-t-2} \}, \\
 g_3 & = & c_3, \\
     & \vdots &  \\
 g_k & = & c_k.
\end{eqnarray*}
 In each of these $4$ cases we have $c(g_1,\dots,g_k,\ooo)(j) = \oo$ for all but possibly $2$ integers $j$.
 We always have
\[ c(g_1,\dots,g_k,\ooo)(i) = c(u_1,\dots, u_k,\oo) \]
 and depending on the choice for $g_1,g_2$ also
\[  \left.\begin{array}{l}
 c(d_i, d_{i-1}, c_3,\dots,c_k,\ooo)(i+t+2) \\
 c(d_i, d_{i-t-2}, c_3,\dots,c_k,\ooo)(i+1) \\
 c(d_{i-t-3}, d_{i-1}, c_3,\dots,c_k,\ooo)(i-1) \\
 c(d_{i-t-3}, d_{i-t-2}, c_3,\dots,c_k,\ooo)(i-t-2) \\
   \end{array} \right\} = c(u_2,u_1,u_3\dots, u_k,\oo). \]
 \emph{In Case 1}, Lemmas~\ref{le:ma} and~\ref{le:c2} yield
 $c(u_2,u_1,u_3\dots, u_k,\oo) = -c(u_1,\dots, u_k,\oo)$. Hence
\[ \sum_{j\in\Z} w(j) = \oo \]
 and the result follows.

 \emph{In Case 2}, $a_1 = a_2$ yields $c(u_1,\dots, u_k,\oo) = c(u_2,u_1,u_3\dots, u_k,\oo)$.
 Since the difference between $i$ and the second index at which $w$ is not $\oo$ is always odd, we obtain
\[ \sum_{j\in\Z} (-1)^j w(j) = \oo. \]
 Again the result follows.
\end{proof}

\begin{lem} \label{le:ghost}
 $g = (\dots,\oo, \underset{0}{e}, \oo,\dots)$ is not in $D$.
\end{lem}

\begin{proof}
 Immediately from Lemma~\ref{le:parity}.
\end{proof}

 Now we verify that the first condition of Lemma~\ref{le:nd} holds for $\nn := 2|B|(2|B|-1)$.
 Let $\varphi\colon \D\to \B$ be a homomorphism with kernel $\ker(\varphi) =: \theta$.
 We will show that $\ker(\varphi|_{D_0})$ has only one block of size greater than $\nn$.

 First we establish a bound on the number of elements $v_{i,i+1}$ that are not congruent to $\ooo$
 modulo $\theta$.

\begin{lem} \label{le:few}
 Let $I := \{ i\in\Z \setsuchthat v_{i,i+1} \not\equiv_\theta \ooo \}$. Then $|I| \leq 2|B|$.
\end{lem}

\begin{proof}
 Suppose that $|I| > 2|B|$. Then there are $p,q,r\in I$ with $p<q<r$ such that
 $d_p \equiv_\theta d_q \equiv_\theta d_r$. Observe that
 \[ \begin{array}{cccccccccccccccccc}
 d_{r} & = & (\dots & \overset{p-t-2}{\oo} & \oo  & \dots & \overset{p}{\oo} &
 {\oo}  & (\dots) & \overset{r}{u_1} & u_2  &\dots & \overset{r+t+2}{u_2} & u_1 &\dots ) \\
 d_{p-t-2} & = & ( \dots & u_1 & u_2              & \dots & u_2 & u_1               & (\dots) &  \oo &\oo & \dots & \oo & \oo & \dots)
\end{array} \]
 where the $(\dots)$ stands for arbitrary (potentially $0$) many terms $\oo$. In particular,
\[ \forall i\in\Z\colon d_{p-t-2}(i) = \oo \text{ or } d_{r}(i) = \oo. \]
 Thus
\[ v_{p,p+1} = f(d_p,d_{p-t-2},c_3,\dots,c_k,\ooo) \equiv_\theta f(d_r,d_{p-t-2},c_3,\dots,c_k,\ooo) = \ooo \]
 which contradicts $p\in I$. Hence $|I| \leq 2|B|$.
\end{proof}

 Next we show that if $L$ is a small enough integer and $J$ a long enough interval with
 $v_{l,l+1} \equiv_\theta \ooo$ for all $l\in J$, then $v_{L,l} \equiv_\theta \ooo$ for all $l\in J$.

\begin{lem} \label{le:large}
 Let $L < \min(I)$, and let $J := \{j,\dots, j+2|B|-2\}$ such that $I\cap J = \emptyset$.
 Then $v_{L,l} \equiv_\theta \ooo$ for all $l\in J$.
\end{lem}

\begin{proof}
 Choose an integer $n \ge 1$ such that $j+2|B|-1-L \leq (2n+1)(t+3)$.
 For $i\in\Z$ define
\[ w_i := m(\dots (m(d_i,d_{i+(t+3)},d_{i+2(t+3)}),\dots), d_{i+(2n-1)(t+3)},d_{i+2n(t+3)}). \]
 Let $-u_2 := m(\oo,u_2,\oo)$. It is straightforward to see
\[ w_i(l) = \begin{cases}
  u_1 & \text{if } l\in\{i, i+(2n+1)(t+3)\}, \\
                   u_2 & \text{if } l = i+1+\lambda(t+3) \text{ for some } \lambda\in\{0,2,4,\dots,2n\}, \\
                   u_2 & \text{if }  l = i+t+2+\lambda(t+3) \text{ for some } \lambda\in\{0,2,4,\dots,2n\}, \\
                   -u_2 & \text{if } l = i+1+\lambda(t+3)  \text{ for some } \lambda\in\{1,3,5,\dots,2n-1\}, \\
                   -u_2 & \text{if } l = i+t+2+\lambda(t+3) \text{ for some } \lambda\in\{1,3,5,\dots,2n-1\}, \\
                   \oo & \text{else.}
\end{cases} \]
That is,
\[ w_i = (\dots,\,\oo,\, \underset{i}{u_1},\, u_2, \,\oo,  \dots,\,\oo, \,u_2,\underset{i+t+3}{\oo},\,-u_2,\, \oo,\dots,\,\oo,\,-u_2,\underset{i+2(t+3)}{\oo},\,u_2,\,\oo,\dots \hspace{8mm} {} \]
\[   {}\hspace{3.4cm} \dots,\,\oo, \,-u_2,\underset{i+2n(t+3)}{\oo},\,u_2,\, \oo,\dots,\,\oo,\, u_2,\underset{i+(2n+1)(t+3)}{u_1}, \, \oo,\dots ). \]
 Let $T := (2n+1)(t+3)$.

Now consider $w_{i}$ and $w_{i-1}$ given by
 \[ \begin{array}{ccccccccccccccccccccc}
w_i&= &( \dots & \oo & \overset{i}{u_1} & u_2  &\dots  &\oo   & u_2 & \overset{i+t+3}{\oo} & -u_2 & \dots 
 &\oo & u_2 &\overset{i+T}{u_1} & \dots ), \\
 w_{i-1}&=& (\dots & u_1 & u_2 & {\oo} &              \dots & u_2 & \oo  & -u_2                 & \oo   & \dots 
  & u_2  & u_1& \oo &\dots).
\end{array} \]
Then
\[ f(w_{i},w_{i-1},c_3,\dots,c_k,\ooo) = v_{i,i+T-1}, \]
where in Case 2, we used that $T-1$ is odd, and hence $v_{i,i+T-1}(i+T-1)= + e = f(u_2, u_1, c_3, \dots, c_k, \oo)$.

 We have integers $p,q,r$ with $j-T+1\leq p < q < r \leq j-T+ 2|B|+1$ such that
 $w_p \equiv_\theta w_q \equiv_\theta  w_r$.
 Since $p+1 < r < p+t$, we have (similar to the argument for Lemma \ref{le:few})
  \[ \begin{array}{ccccccccccccccccccccc}
  w_{p-1}& = & ( \dots & u_1 & u_2               &    \oo   &               (\dots)  & \oo               & \oo                          & (\dots)  \\
  w_{r}  & = & (\dots & \oo & \underset{p}{\oo} & \underset{\le q}{\oo}  & (\dots)  &\underset{r}{u_1}  & \underset{\le (p-1)+t+1}{u_2} & (\dots) \\  \\
          &     & u_2                       & \oo &-u_2                     & (\dots) & \oo                  & \oo  & \oo                                   &(\dots)  \\
          &    & \underset{(p-1)+t+2}{\oo} & \oo & \underset{\le q+t+2}{\oo}& (\dots) & \underset{r+t+2}{u_2}&  \oo & \underset{ \le (p-1)+2(t+3)-2}{-u_2} & (\dots) \\  \\
          &    & u_2                       & u_1                      & (\dots) & \oo & \oo                 & \dots)                            \\
          &    & \underset{(p-1)+T-1}{\oo} & \underset{\le q+T-1}{\oo}& (\dots) & u_2& \underset{r+T}{u_1} & \dots ) \\  \\
\end{array} \]
 We observe that
\[ \forall i\in\Z\colon w_{p-1}(i) = \oo \text{ or } w_{r}(i) = \oo. \]
 Thus
\[ v_{p,p+T-1} = f(w_{p},w_{p-1},c_3,\dots,c_k,\ooo) \equiv_\theta f(w_{r},w_{p-1},c_3,\dots,c_k,\ooo) = \ooo. \]
 Since $p \leq j-T+2|B|-1\leq L$ and $v_{l,l+1}\equiv_\theta \ooo$ for all $l\leq L$, this yields
\[ v_{L,p+T-1} \equiv_\theta \ooo. \]
 Note that $p+T-1$ is in $J$. Since $v_{l,l+1}\equiv_\theta \ooo$ for all $l\in J$, we finally obtain
\[ v_{L,l} \equiv_\theta \ooo \text{ for all } l\in J. \]
\end{proof}

\begin{lem}
 Let $L$ be an even negative integer with $L < \min(I)$.
 Then
\[ |\{i\in\N \setsuchthat  v_{0,i} \not\equiv_\theta -v_{L,0} \}| \leq 2|B|(2|B|-1). \]
\end{lem}

\begin{proof}
 Let $i\in\N$. Note that
\[ v_{L,i} = v_{L,0}+v_{0,i}. \]
 Assume $v_{0,i} \not\equiv_\theta -v_{L,0}$. Then $v_{L,i}\not\equiv_\theta \ooo$.
 So, by Lemma~\ref{le:large}, we have
 $j\in I$ such that $|i-j| \leq |B|-1$.
 Together with Lemma~\ref{le:few} this yields the result.
\end{proof}

 Finally all assumptions of Lemma~\ref{le:nd} are satisfied for $\nn = 2|B|(2|B|-1)$ and $\D$ and $D_0$ as above.
 It follows that $\B$ is not dualizable. Theorem~\ref{th:sn} is proved.

\begin{proof}[Proof of Theorem~\ref{th:nd}]
 Let $\A$ be a finite non-abelian nilpotent algebra of finite type that splits into a direct product of
 prime power algebras. Then $\A$ is supernilpotent by~\cite[Lemma 7.6]{AM:SAHC}.
 Hence $\A$ satisfies the assumptions of Theorem~\ref{th:sn} with $\alpha := 1$, and the result follows.
\end{proof}

\section{A clone theoretic characterization of duality} \label{sec:clone}

 In the next section, we will give an example of a dualizable algebra that limits how far
 Theorem~\ref{th:sn} can be generalized. We will show duality by a novel approach using clone theory
 that was suggested by Willard in a conference talk~\cite{Wi:FUPC} and further developed
 by Davey, Pitkethly, and Willard in the paper~\cite{DP:LAE}.
 We note that in~\cite{DP:LAE} the authors
 develop an extension of the duality theory from~\cite{CD:NDWA}; see the appendix for a
 discussion of the clone theoretic approach in this new so-called {\em symmetric setting}
 and explanations about the (minor) differences between the theories.
 As explained in the appendix to~\cite{DP:LAE}, results stated in one setting can be readily
 translated to the corresponding results in the other setting.

 In the present section we state Willard's clone theoretic condition that is sufficient for 
 dualizability in Corollary~\ref{cor:frd}. For the convenience of the reader, we give a proof
 based on wellknown results from the book~\cite{CD:NDWA}, namely the Third Duality Theorem
 and the Duality Compactness Theorem; see the appendix for a development of the corresponding
 results in the symmetric setting of~\cite{DP:LAE}.

We briefly recall the definition of duality from~\cite{CD:NDWA}.
Let $\A_0$ be a finite
 algebra with universe $A_0$, and consider a topological structure
 $\undertilde{\A}_1 := \algop{A_0}{\mathcal{F},\mathcal{R},\tau}$ on
$A_0$, where $\mathcal{F}$ is a set of total or  partial operations, $\mathcal{R}$
 is a set of relations,
and $\tau$ is the discrete topology, such that each operation of $\A$ preserves
each relation $R \in \mathcal{R}$ and the graph of each $F \in \mathcal{F}$.
 Then  $\undertilde{\A}_1$ is called an \emph{alter ego} of $\A_0$.

The aim of duality theory is to find suitable choices of $\mathcal{F}$ and $\mathcal{R}$ in order to set up a
dual representation between
two different categories. One corresponds to the quasi-variety $\mathcal{A}:=\ISP{\A_0}$ generated by $\A_0$, consisting 
 of all isomorphic copies of subalgebras of powers of $\A_0$. The other corresponds to the topological quasi-variety $\mathcal{X}:=\ISCP{\undertilde{\A}_1}$
of isomorphic copies of closed substructures of powers of $\undertilde{\A}_1$, where powers are taken over non-empty index sets. The morphisms of the categories are the homomorphisms and continuous homomorphisms, respectively,
among their objects.

We can set up mappings $D: \mathcal{A} \rightarrow \mathcal{X}$ and
$E:\mathcal{X} \rightarrow \mathcal{A}$. For $\A \in \mathcal{A}$
let $D(\A)$ be the substructure of $\undertilde{\A}_1^A$ whose universe consists of all homomorphisms from
$\A$ to $\A_0$. For $\X \in \mathcal{X}$ let $E(\X)$ be the subalgebra of $\A_0^X$ whose universe consists of all
continuous homomorphisms from $\X$ to $\undertilde{\A}_1$ (we remark in passing that $D$ and $E$ can be extended to contravariant functors between $\mathcal{A}$ and $\mathcal{X}$).

Now for each $\A \in \mathcal{A}$ we have a natural embedding $e_{\A}: \A \rightarrow ED(\A)$ given
by evaluation, i.e. $e_{\A}(a)$ is given by $h \mapsto h(a)$ for each continuous homomorphism $h \in D(\A)$.

We say that $\undertilde{\A}_1$ {\em dualizes} $\A_0$ if $e_A$ is an isomorphism for
each $\A \in \mathcal{A}$. $\A_0$ is {\em dualizable} if there is a structure $\undertilde{\A}_1$ that dualizes
$\A_0$.  We say that $\undertilde{\A}_1$ {\em dualizes} $\A_0$ {\em at the finite level} if $e_A$ is an isomorphism for
each finite $\A \in \mathcal{A}$.

The next definitions will be used in providing a clone theoretic approach to dualizability.
 Let $\A$ be an algebra.
 A subset 
 $D$ of some finite power $A^k$ of $A$ is \emph{c.a.d. (conjunct-atomic definable }~\cite{DP:LAE})
 over $\A$ if it is definable by a conjunction of atomic formulas over $\A$. That is, $D$ is c.a.d.
 over $\A$ if there exist $f_1,\dots,f_l,g_1,\dots,g_l\in\Clo_k(\A)$ such that
\[ D = \{ x\in A^k \setsuchthat f_1(x) = g_1(x), \dots, f_l(x) = g_l(x) \}. \]
 Note that the empty set $\emptyset$ may be c.a.d. over $\A$.
 In~\cite[p. 66]{CD:NDWA} c.a.d. relations are called term closed.

 Let $R\subseteq A^n$, and for $D\subseteq A^k$ let $f$ be a partial operation $f\colon D\to A$. We can extend
  $f$ to a partial operation $f^{A^n}$ on $A^n$ by evaluating $f$ coordinatewise. 
   We say that $f$
 \emph{preserves} $R$ if
\[ \forall r_1,\dots, r_k\in R\colon f^{A^n}(r_1,\dots, r_k) \in R \text{ whenever defined.} \]

 We denote the set of all restrictions of term operations on $\A$ to c.a.d. domains by $\Clo_{cad}(\A)$.
 We say that the partial clone $\Clo_{cad}(\A)$ is \emph{finitely related} if there exist a finite set
 $\R=\{R_1,\dots,R_l\}$ of subuniverses of finite powers of $\A$ such that the following are equivalent
 for every partial operation $f$ on $A$ with c.a.d. domain over $\A$:
\begin{enumerate}
\item \label{it:cad1} $f$ preserves the relations $R_1,\dots,R_l$,
\item \label{it:cad2} $f\in\Clo_{cad}(\A)$.
\end{enumerate}
 We note that the implication~\eqref{it:cad2}~$\Rightarrow$~\eqref{it:cad1} is trivially satisfied
 because $R_1,\dots,R_l$ are subuniverses of powers of $\A$.

We will need the following variants of well known results.

\begin{thm}[cf. Third Duality Theorem, 3.1.6 in~\cite{CD:NDWA}]
Let $\R:=\{R_1, \dots, R_l\}$ be a finite set of relations on $A$, the universe of the finite algebra $\A$,
 and let $\undertilde{\A}:=\algop{A}{\emptyset,\R,\tau}$ be 
 with $\tau$ the discrete topology on $A$.
 The following are equivalent:
\begin{enumerate}
\item \label{it:tdt1}
 The structured space $\undertilde{\A}$ dualizes $\A$ at the finite level.
\item \label{it:tdt2}
 For any $k\in\N$ and any substructure $\undertilde{\D}$ of $\undertilde{\A}^k$ whose universe $D$ is c.a.d.
 over $\A$, every morphism $\undertilde{\D} \to \undertilde{\A}$ extends to a total $k$-ary term function on~ $\A$.
 \item \label{it:tdt3}
 $\Clo_{cad}(\A)$ is finitely related by the relations $R_1,\dots,R_l$.
\end{enumerate}
\end{thm}

\begin{proof}
\eqref{it:tdt1}~$\Leftrightarrow$~\eqref{it:tdt2} is a special case of the finite level case of the Third
 Duality Theorem  3.1.6 in~\cite{CD:NDWA}.
 For~\eqref{it:tdt2}~$\Rightarrow$~\eqref{it:tdt3} consider a partial operation $f\colon D\to A$ with
 domain $D\subseteq A^k$ c.a.d. over $\A$ such that $f$ preserves the relations in $\R$. Then $D$
 induces a substructure $\undertilde{\D}$ of $\undertilde{\A}^k$. As $D$ is finite and $\tau$ discrete,
and $f$ preserves all relations $R$ in $\R$,
 $f$ is 
 actually a morphism from $\undertilde{\D}$ to $\undertilde{\A}$ in the sense of~\cite[p. 21]{CD:NDWA}. By~\eqref{it:tdt2} $f$ extends
 to a total term operation on $\A$, that is, $f\in\Clo_{cad}(\A)$. Hence we have~\eqref{it:tdt3}.

 The converse implication~\eqref{it:tdt3}~$\Rightarrow$~\eqref{it:tdt2} follows similarly:
 Let $\undertilde{\D}$ be a substructure of $\undertilde{\A}^k$ whose universe is c.a.d. over $\A$,
 and let $f\colon\undertilde{\D}\to\undertilde{\A}$ be a morphism. Then $f$ preserves the relations $R_1,\dots,R_l$
 and $f\in\Clo_{cad}(\A)$ by~\eqref{it:tdt3}. Thus $f$ extends to a term operation on $\A$, and
 we have~\eqref{it:tdt2}.
 $\Clo_{cad}(\A)$. No
\end{proof}

\begin{thm}[Duality Compactness Theorem, 2.2.11 in~\cite{CD:NDWA}; independently Willard~\cite{Wi:NTFP}, Zadori~\cite{Za:NDVA}]
Let $\A$ be a finite algebra. If the structure $\undertilde{\A}$ is of finite type and dualizes $\A$ at the finite level, then
$\undertilde{\A}$ dualizes~$\A$.
\end{thm}

 Combining the above results (equivalently, combining the version of Duality Lemma 4.1
 from~\cite{DP:LAE} for the so-called \emph{usual setting} with the Duality Compactness
 Theorem), we immediately get:

\begin{cor}[Willard~\cite{Wi:FUPC}] \label{cor:frd}
 Let $\A$ be a finite algebra. If $\Clo_{cad}(\A)$ is finitely related by relations $R_1,\dots,R_l$,
 then $\A$ is dualized by the finitary structure
 $\undertilde{\A} := \algop{A}{\emptyset,\{R_1,\dots,R_l\},\tau}$ with $\tau$ the discrete topology on $A$. 
\end{cor}

\section{Nilpotent and dualizable} \label{sec:dualizable}

 We give an example that shows that in
 Theorem~\ref{th:sn} the condition that $\A$ has a non-abelian, supernilpotent congruence cannot
 be removed. In particular dualizable non-abelian nilpotent Mal'cev algebras exist.

 Throughout the remainder of this section let
$$\A := \algop{\Z_4}{+, 1, \{ 2x_1\dots x_k \setsuchthat k\in\N \}}.$$
 We note that $\A$ is nilpotent with center $\equiv_2$ but not supernilpotent. Observe that every
 reduct of finite type of $\A$ with group operation
  is inherently non-dualizable by Theorem~\ref{th:nd}.
 Still we have the following result:

\begin{thm} \label{th:inf}
 $\A := \algop{\Z_4}{+, 1, \{ 2x_1\dots x_k \setsuchthat k\in\N \}}$ is dualizable by the alter ego
$\undertilde{\A}:= \algop{A}{\emptyset,\R,\tau}$, where $\R$ is the set of
all subuniverses of $\A^4$.
\end{thm}

 Before proving the theorem we describe the term operations on $\A$.
 For $k\in\N$ and $v\in\Z_4^k$, define
\[ c_v\colon \Z_4^k\to \Z_4, x\mapsto \begin{cases} 2 & \text{if } x\in v+2\Z_4^k, \\
                                                   0 & \text{else.} \end{cases} \]

\begin{lem} \label{le:term}
 Let  $f\colon \Z_4^k\to \Z_4$ with $f(0,\dots, 0) = 0$. Then $f\in\Clo_k(\A)$ iff
 $\exists \lambda_1,\dots,\lambda_k\in\Z_4, \exists v_1,\dots, v_l$ $ \in \Z_4^k$ such that $\forall x\in\Z_4^k\colon$
\begin{equation} f(x) = \sum_{i=1}^k \lambda_i x_i + c_{v_1}(x)+\dots+c_{v_l}(x). \label{mod} \end{equation}
\end{lem}

\begin{proof}
 Any $f \in\Clo_k(\A)$ with $f(0,\dots,0)=0$ can be written as
$$  f(x) = \sum_{i=1}^k \lambda_i x_i +g(x),$$
 where $g(x)$ is a sum of terms
 $2x_{j_1}\dots x_{j_n}$ for some $n\in\N$ and $j_1,\dots,j_n\in\{1,\dots,k\}$.
 It follows that $g(x)=g(x+y)$ for all $y \in 2\Z_4^k$. Pick a set $v_i$
of representatives for those residue classes of $\Z_4^k \mod 2 \Z_4^k$ which are mapped to $2$ by $g$. Expression
(\ref{mod}) follows.

Conversely, let $f$ be of the form \eqref{mod}. If $v_i=(w_1, \dots, w_k)$, then $\forall x_1,\dots,x_k\in A$
 $$c_{v_i}(x_1,\dots,x_k)=2(1+x_1+w_1)\cdots(1+x_k+w_k). $$
 Hence all $c_{v_i}$ are in $\Clo_k(\A)$, and hence $f \in \Clo_k(\A)$.
\end{proof}

 Next we determine the c.a.d. relations over $\A$.

\begin{lem}  \label{le:cad}
 Let $D\subseteq \Z_4^k$ with $(0,\dots,0)\in D$, let $B := 2\Z_4^k$.
 Then $D$ is c.a.d. over $\A$ if and only if there exists a subgroup $U$ of $\algop{B}{+}$ and
 $v_1,\dots, v_l \in \Z_4^k$ such that $v_i - v_j \not\in B$ whenever $i\neq j$, $2v_i\in U$ for all
 $i$, and
\begin{equation} \label{eq:cadD}
 D  = v_1+U \cup\dots\cup v_l+U.
\end{equation}
\end{lem}

\begin{proof}
 Assume $D$ is c.a.d. We have $f_1,\dots,f_n\in\Clo_k(\A)$ such that
 $D = \{x\in\Z_4^k \setsuchthat f_1(x) = 0, \dots, f_n(x) = 0 \}$.
 By Lemma~\ref{le:term} we have a $k\times n$-matrix $H$ over $\Z_4$ and vectors $b_w\in 2\Z_4^n$
 for $w\in\{0,1\}^k$ such that
\[ D = \bigcup_{w\in\{0,1\}^k} \{ x\in w+B \setsuchthat H\cdot x = b_w \}. \]
 Let $U := \{ x\in B \setsuchthat H\cdot x = 0 \}$. Since $(0,\dots,0)\in D$, we have
$b_{(0,\dots,0)}=(0,\dots,0)$ and $D\cap B = U$.

 Let $w\in\{0,1\}^k$ such that $D\cap w+B \neq \emptyset$, say $v\in D\cap w+B$. It follows that
 $D\cap w+B = v+U$. Since $H\cdot v = b_w \in 2\Z_4^n$, we have $H\cdot(2v) = 0$ and $2v\in U$.
 Hence $D$ is as in \eqref{eq:cadD}.

 The converse implication of the lemma is now straightforward.
\end{proof}

 Finally we show that $\Clo_{cad}(\A)$ is finitely related.

\begin{lem} \label{le:fr}
 $\Clo_{cad}(\A)$ is the set of partial operations with c.a.d. domain over $\A$ that preserve
 all subuniverses of $\A^4$.
\end{lem}

\begin{proof}
 Let $D \subseteq \Z_4^k$ be c.a.d. over $\A$, and let $f\colon D\to\Z_4$ preserve all subuniverses of $\A^4$.
 We will show that $f$ is the restriction of a term operation on $\A$.
 Since all constants are term operations on $\A$, we may assume w.l.o.g. that $(0,\dots,0)\in D$ and
 $f(0,\dots,0) = 0$.

 Let $U := D\cap 2\Z_4^k$. By Lemma~\ref{le:cad} we have $v_1,\dots, v_l \in \Z_4^k$ such that
 $v_i - v_j \not\in 2\Z_4^k$ whenever $i\neq j$, $2v_i\in U$ for all $i$, and
 $D  = v_1+U \cup\dots\cup v_l+U$.
 Note that
\[ M := \{ \mtt{a}{b}{a+2c}{b+2c} \setsuchthat a,b,c\in A \} \]
 is a subuniverse of $\A^{2\times 2}$ and hence preserved by $f$.
 Thus for all $x\in D, u\in U$ with $x+u\in D$ we have
\[ \mtt{f(0,\dots,0)}{f(x)}{f(u)}{f(x+u)} \in M. \]
 Hence
\begin{equation} \label{eq:hom}
 \forall x\in D, u\in U\colon x+u\in D \Rightarrow f(x+u) = f(x)+f(u).
\end{equation}
 In particular $f|_U$ is a group homomorphism from $U$ to $\Z_4$.
 Hence we have $t\in\Clo_k(\algop{\Z_4}{+})$ such that $f|_U = t|_U$.
 For $x\in D$ define
\[ g(x) := f(x)-t(x). \]
 Let $i\in\{1,\dots,l\}$. From~\eqref{eq:hom} it follows that
\begin{equation} \label{eq:const}
 \forall u\in U\colon g(v_i+u) = g(v_i).
\end{equation}
 We claim that
\begin{equation} \label{eq:in2Z4}
 g(v_i) \in 2\Z_4.
\end{equation}
 For the proof consider
\[ h_i\colon \Z_4\to\Z_4, x\mapsto g(x v_i). \]
 Note that $h_i$ is indeed a total unary operation on $\Z_4$ since $2v_i\in U$ implies that
 $\Z_4 v_i \subseteq D$.
 Since $\Clo_1(\A)$ embeds into $\A^4$, we have that $g$ preserves $\Clo_1(\A)$ and consequently
 $h_i\in\Clo_1(\A)$.
 Now $h_i(0) = g(0) = 0$ and $h_i(2) = g(2v_i) = 0$. As $h_i$ is a term operation on $\A$,
 this implies either $h_i(x) = 0$ for all $x\in\Z_4$ or $h_i(x) = 2x$ for all $x\in\Z_4$.
 In any case $h_i(1) = g(v_i)$ is in $2\Z_4$ which proves~\eqref{eq:in2Z4}.

 From~\eqref{eq:const} and~\eqref{eq:in2Z4} we obtain that for all $x\in D$
\[ g(x) = \sum\{ c_{v_i}(x) \setsuchthat g(v_i) = 2, i\in\{1,\dots,l\} \}. \]
 Hence $f$ is the restriction of the term operation
 $t+\sum\{ c_{v_i} \setsuchthat g(v_i) = 2, i\in\{1,\dots,l\} \}$.
\end{proof}

\emph{Proof of Theorem~\ref{th:inf}.}
 Since $\Clo_{cad}(\A)$ is finitely related by the elements of $\mathcal{R}$ by Lemma~\ref{le:fr},
 $\undertilde{\A}$ dualizes $\A$ by Corollary~\ref{cor:frd}.
\qed

\section{Problems}

 Pontryagin duality yields that finite abelian groups are dualizable. In general,
 abelian algebras in congruence modular varieties are polynomially equivalent to modules over rings.
 Although the structure of these algebras is well understood, to our knowledge the following is
 still open.

\begin{pro}
 Is every finite abelian algebra in a congruence modular variety dualizable?
 Is every finite module over a ring dualizable?
\end{pro}

 More generally we would also like a characterization of all nilpotent algebras that are dualizable.
 Is supernilpotence the only obstacle for dualizability? Can we prove some kind of converse to
 Theorem~\ref{th:sn}?

\begin{pro}
 Are the following equivalent for every finite nilpotent algebra $\A$?
\begin{enumerate}
\item $\A$ is dualizable.
\item For every subalgebra $\B$ of $\A$, all supernilpotent congruences of $\B$ are abelian.
\end{enumerate}
\end{pro}

\section{Appendix}

In \cite{DP:LAE}, Davey et al. develop a new approach to dualities that extends the theory (on the algebraic side)
to structures with partial operations and relations, and also places it in a slightly more restricted setting
(termed the \emph{symmetric setting}).  An extension to the duality
theory from \cite{CD:NDWA} is discussed in an appendix and termed the \emph{usual setting}.
The differences between the various settings are restricted to trivial members of the involved categories,
however they do effect part of the argument from Section \ref{sec:clone}. We have therefore decided to base the arguments in that section on \cite{CD:NDWA}, and to included this appendix showing how to derive the same
conclusions in the \emph{symmetric setting} of \cite{DP:LAE}.

 We need to make slight adjustments to the definitions from Section \ref{sec:clone}.
In order to state the result in the language of \cite{DP:LAE} we will formulate it for
structures, although we will only use structures that are effectively algebras.
We will also slightly change the definition of $\Clo_{cad}$ to $\Clo_{cad}^*$. This change is actually
not necessary in formulating the result, but reflects that the symmetric setting avoids
partial operations with empty domains. Note
that the modifier ${}^*$ is not from \cite{DP:LAE} but is used here to clearly distinguish the two definitions.

 Duality in \cite{DP:LAE} is defined not for algebras but for base structures.   A \emph{base structure} is a
 finite non-empty
 structure $\algop{A}{\mathcal{F},\R}$ where $\mathcal{F}$ is a set of partial operations on $A$ and
 $\R$ a set of relations on $A$, where (in the symmetric setting)
 all functions are non-nullary and have non-empty domain, and all relations are non-empty.

For our purposes, let $\algop{A_0}{\mathcal{F}}$ be an
algebra without  nullary operations. This modification is inconsequential,
as we will also (see below) exclude the empty structure from our modified version of
$\mathcal X$.
Hence we may replace any constants by the corresponding
unary function with constant image. We will identify $\A_0$ with the base structure $\algop{A_0}{\mathcal{F},\emptyset}$,
keeping the name $\A_0$ for both. We will leave it to the reader to check that dualizability of the algebra
$\A_0$, as defined in Section \ref{sec:clone} corresponds to dualizability of the base structure $\A_0$ in the
so called 
\emph{usual setting} of \cite{DP:LAE}.

Duality in the context of the {\em symmetric setting} is characterized by replacing
the category $\mathcal{A}$
with a slightly more restricted versions of itself.
Let
$\mathcal{A}^*$ be the category $\ISPP{\A_0}$ of all base structures isomorphic to substructures
 of powers of $\A_0$ over non-empty index sets, together with their morphisms.
So when considered as a class of algebras, $\mathcal{A}^*$ will either equal $\mathcal{A}$ or can be obtained
 from $\mathcal{A}$ by removing all one-element algebras. Similarly, on the topological side we replace $\mathcal X$ with $\mathcal{X}^*$, which is obtained from $\mathcal X$ by removing the empty structure  and
 its morphisms (if present).
 We define the functors $D$ and $E$ as before, and we get a notion of duality or finite level duality
 in the symmetric setting by restricting the corresponding definitions to structures from $\mathcal{A}^*$.

As before we define a subset $D\subseteq A_0^k$ for $k \ge 1$ to be
 \emph{c.a.d. (conjunct-atomic definable)} over $\A_0$ if it is definable
 by a conjunction of atomic formulas over $\A_0$.
 The set of all finitary {\em non-empty} relations that are c.a.d. over
 $\A_0$ is called $\Def_{ca} (\A)$.
We denote the set of all restrictions of term operations on $\A$ to domains in $\Def_{ca}(\A)$ by $\Clo_{cad}^*(\A)$
(note that in \cite{DP:LAE}, the elements of $\Clo_{cad}^*(\A)$ are called \emph{structural functions}).

 Let $R\subseteq A_0^n$ be non-empty.
 For non-empty $D\subseteq A_0^k$ a partial operation $f\colon D\to A_0$ \emph{preserves} $R$ if
\[ \forall r_1,\dots, r_k\in R\colon f^{A^n}(r_1,\dots, r_k) \in R \text{ whenever defined,} \]
where $f^{A^n}$ is defined as in Section \ref{sec:clone}.
 We say $\Clo_{cad}^*(\A_0)$ is \emph{finitely related} if there exist finitely many subuniverses
 $R_1,\dots,$ $R_l$ of finite powers of $\A_0$ such that the following are equivalent for every partial
 operation $f$ on $A$ with non-empty c.a.d. domain over~$\A_0$:
\begin{enumerate}
\item $f$ preserves the relations $R_1,\dots,R_l$,
\item $f\in\Clo_{cad}^*(\A_0)$.
\end{enumerate}

We are now able to state a sufficient condition for dualizability in the sense of~\cite{DP:LAE}. 
Instead of a direct derivation in this setting (in a similar fashion
 to Section \ref{sec:clone} with the 
role of the Third Duality Theorem replaced by Theorem 4.1 of \cite{DP:LAE}), we will instead translate 
between the theories.

\begin{cor} \label{le:frd2}
 Let $\A_0$ be a finite algebra without nullary operations.
 Assume that $\Clo_{cad}^*(\A_0)$ is finitely related by $R_1,\dots,R_l$.
 Then $\A_0$ is dualized in the symmetric setting of~\cite{DP:LAE} by
 $$\algop{A_0}{\emptyset,\{R_1,\dots,R_l\}, \tau},$$ where $\tau$ is the discrete topology.
\end{cor}

\begin{proof}
As the empty operation is compatible with all relations, $\Clo_{cad}(\A_0)$ is also finitely related by
 $R_1,\dots,R_l$
and hence is dualized by $\algop{A_0}{\emptyset,\{R_1,\dots,R_l\}, \tau}$  in the usual setting by Corollary \ref{cor:frd}. By Lemma A.4 of \cite{DP:LAE} we get a duality in the symmetric setting if we remove all empty
operations from $\A_0$ and replace all nullary operations of $\algop{A_0}{\emptyset,\{R_1,\dots,R_l\}, \tau}$
with unary ones. However, in our case these changes preserve both the algebraic and the topological structure. The result follows.
\end{proof}

\noindent {\bf Acknowledgment}
The first author  has received funding from the
European Union Seventh Framework Programme (FP7/2007-2013) under
grant agreement no.\ PCOFUND-GA-2009-246542 and from the Foundation for
Science and Technology of Portugal.

 The second author acknowledges support from the Portuguese Project PEst-OE/MAT/UI0143/ 2011
 of CAUL
financed by FCT
 and FEDER.

The authors would like to thank Brian Davey for his helpful suggestions,
Maria Jo\~ao Gouveia and Lu\'is Sequeira for their support in conducting this research and writing this article, as well as Michael Kinyon for his suggestions regarding literature.


\begin{thebibliography}{99}
\bibitem{AM:SAHC}
E. Aichinger and N. Mudrinski.
\newblock Some applications of higher commutators in {M}al'cev algebras.
\newblock {\em Algebra Universalis}, 63(4):367--403, 2010.

\bibitem{Br:CTL}
R. H. Bruck.
\newblock Contributions to the Theory of Loops.
\newblock {\em Trans. Amer. Math. Soc.} 60(2): 245--354, 1946.


\bibitem{Bu:NFMA}
A. Bulatov.
\newblock On the number of finite Mal'cev algebras.
In: Proceedings of the
Dresden Conference 2000 (AAA 60) and the Summer School 1999
\newblock {\em Contr.\ Gen.\ Alg.} 13, 41--54, 2001.

\bibitem{CD:NDWA}
D.~M. Clark and B.~A. Davey.
\newblock {\em Natural dualities for the working algebraist}, volume~57 of {\em
  Cambridge Studies in Advanced Mathematics}.
\newblock Cambridge University Press, Cambridge, 1998.

\bibitem{CD:CDT}
D.~M. Clark, B.~A. Davey, and J.~G. Pitkethly.
\newblock The complexity of dualisability: three-element unary algebras.
\newblock {\em Internat. J. Algebra Comput.}, 13(3):361--391, 2003.

\bibitem{CI:NDQ}
D.~M. Clark, P.~M. Idziak, L.~R. Sabourin, Cs. Szab{\'o}, and
  Ross Willard.
\newblock Natural dualities for quasivarieties generated by a finite
  commutative ring.
\newblock {\em Algebra Universalis}, 46(1-2):285--320, 2001.
\newblock The Viktor Aleksandrovich Gorbunov memorial issue.

\bibitem{DHM:NAOG}
B.~A. Davey, L. Heindorf, and R. McKenzie.
\newblock Near unanimity: an obstacle to general duality theory.
\newblock {\em Algebra Universalis} 33:428--439, 1995.

\bibitem{DP:LAE}
B.~A. Davey, J.~G. Pitkethly, and R.~Willard.
\newblock The lattice of alter egos.
\newblock {\em Internat. J. Algebra Comput.}, 22(1):36, 2012.

\bibitem{FM:CTFC}
R.~Freese and R.~N. McKenzie.
\newblock {\em Commutator Theory for Congruence Modular Varieties}, volume 125
  of {\em London Math. Soc. Lecture Note Ser.}
\newblock Cambridge University Press, 1987.

\bibitem{Ke:CMVW}
K.~A. Kearnes.
\newblock Congruence modular varieties with small free spectra.
\newblock {\em Algebra Universalis}, 42(3):165--181, 1999.

\bibitem{Ma:MASC}
P.~Mayr.
\newblock {Mal'cev algebras with supernilpotent centralizers.}
\newblock {\em Algebra Universalis}, 65(2):193--211, 2011.

\bibitem{MMT:ALVV}
R.~N. McKenzie, G.~F. McNulty, and W.~F. Taylor.
\newblock {\em Algebras, lattices, varieties, Volume {I}}.
\newblock Wadsworth \& Brooks/Cole Advanced Books \& Software, Monterey,
  California, 1987.

\bibitem{QS:NGAN}
R.~Quackenbush and Cs. Szab{\'o}.
\newblock Nilpotent groups are not dualizable.
\newblock {\em J. Aust. Math. Soc.}, 72(2):173--179, 2002.

\bibitem{Sz:FNR}
Cs. Szab{\'o}.
\newblock Finite nilpotent rings are not dualizable.
\newblock {\em Algebra Universalis}, 42(4):293--298, 1999.

\bibitem{Ve:PGLG}
A. Vesanen.
\newblock On p-groups as loop groups.
\newblock {\em Arch. Math. (Basel)}, 61:1--6, 1993.

\bibitem{Wi:NTFP}
R. Willard.
\newblock New tools for proofing dualizability.
In: {\em Dualities, Interpretability and Ordered Structures} (J. Vaz de Carvalho and I. Ferreirim, eds),
 Centro de Algebra da Universidade de Lisboa: 69--74, 1999.

\bibitem{Wi:FUPC}
 R.~Willard.
\newblock {\em Four unsolved problems in congruence permutable varieties}.
\newblock Talk at the Conference on Order, Algebra, and Logics, Nashville, 2007.


\bibitem{Za:NDVA}
L. Zadori.
\newblock Natural Duality via a finite set of relations.
\newblock {\em Bull. Aust. Math. Soc.} 51:469--478, 1995.

\end{thebibliography}
\end{document}